%% file: hottcats.tex
\pdfoutput=1

\documentclass{amsart}
\usepackage[utf8]{inputenc}
\usepackage{amssymb,amsmath,amsthm,stmaryrd,mathrsfs}
\usepackage{enumitem,mathtools,xspace,xcolor}
\definecolor{darkgreen}{rgb}{0,0.45,0}
\definecolor{dkblue}{rgb}{0,0.1,0.5}
\definecolor{lightblue}{rgb}{0,0.5,0.5}
\definecolor{dkgreen}{rgb}{0,0.4,0}
\definecolor{dk2green}{rgb}{0.4,0,0}
\definecolor{dkviolet}{rgb}{0.6,0,0.8}
\usepackage{listings}

\usepackage[pagebackref,colorlinks,citecolor=darkgreen,linkcolor=darkgreen]{hyperref}
\usepackage[all,2cell]{xy}
\UseAllTwocells
\usepackage{braket}

\newcommand{\jdeq}{\equiv}      

\newcommand{\defeq}{\coloneqq}  

\makeatletter
\def\prd#1{{\textstyle\prod}(#1)\@ifnextchar\bgroup{\;\prd}{,}}

\def\fall#1{{\forall}(#1)\@ifnextchar\bgroup{\;\fall}{,}}

\def\sm#1{{\textstyle\sum}(#1)\@ifnextchar\bgroup{\;\prd}{,}}
\makeatother
\let\setof\Set    

\newcommand{\id}[3][]{\ensuremath{#2 =_{#1} #3}\xspace}

\newcommand{\refl}[1]{\ensuremath{\mathsf{refl}_{#1}}\xspace}

\newcommand{\ct}{\mathrel{\raisebox{.5ex}{$\centerdot$}}}

\newcommand{\opp}[1]{\mathord{{#1}^{-1}}}
\let\rev\opp

\newcommand{\trans}[2]{\ensuremath{{#1}_{*}\!\left({#2}\right)}\xspace}

\newcommand{\map}[2]{\ensuremath{{#1}\left({#2}\right)}\xspace}

\newcommand{\idfunc}[1][]{\ensuremath{\mathsf{id}_{#1}}\xspace}

\newcommand{\eqv}[2]{\ensuremath{#1 \simeq #2}\xspace}

\newcommand{\type}{\ensuremath{\mathsf{Type}}\xspace}
\renewcommand{\set}{\ensuremath{\mathsf{Set}}\xspace}
\newcommand{\prop}{\ensuremath{\mathsf{Prop}}\xspace}

\newcommand{\trunc}[2]{\Vert #2\Vert_{#1}}
\newcommand{\ttrunc}[2]{\big\Vert #2\big\Vert_{#1}}

\def\pizero{\trunc0}
\newcommand{\brck}[1]{\trunc{}{#1}}
\newcommand{\bbrck}[1]{\ttrunc{}{#1}}

\newcommand{\uset}{\ensuremath{\underline{\set}}\xspace}

\makeatletter
\def\defthm#1#2{%
  \newtheorem{#1}{#2}[section]%
  \expandafter\def\csname #1autorefname\endcsname{#2}%
  \expandafter\let\csname c@#1\endcsname\c@thm}

\newtheorem{thm}{Theorem}[section]

\defthm{cor}{Corollary}
\defthm{lem}{Lemma}
\theoremstyle{definition}
\defthm{defn}{Definition}
\theoremstyle{remark}
\defthm{rmk}{Remark}
\defthm{eg}{Example}
\defthm{egs}{Examples}
\defthm{notes}{Notes}

\let\c@equation\c@thm
\numberwithin{equation}{section}

\setitemize[1]{leftmargin=2em}
\setenumerate[1]{leftmargin=*}

\newcommand{\inv}[1]{{#1}^{-1}}
\newcommand{\idtoiso}{\ensuremath{\mathsf{idtoiso}}\xspace}
\newcommand{\isotoid}{\ensuremath{\mathsf{isotoid}}\xspace}
\newcommand{\op}{^{\textrm{op}}}
\newcommand{\y}{\ensuremath{\mathbf{y}}\xspace}

\title{Univalent categories and the Rezk completion}
\author{Benedikt Ahrens}
\author{Krzysztof Kapulkin}
\author{Michael Shulman}
\date{\today}
\begin{document}

\begin{abstract}
  We develop category theory within Univalent Foundations, which is a foundational system for mathematics based on a homotopical interpretation of dependent type theory.
  In this system, we propose a definition of ``category'' for which equality and equivalence of categories agree.
  Such categories satisfy a version of the Univalence Axiom, saying that the type of isomorphisms between any two objects is equivalent to the identity type between these objects; we call them ``saturated'' or ``univalent'' categories.
  Moreover, we show that any category is weakly equivalent to a univalent one in a universal way.
  In homotopical and higher-categorical semantics, this construction corresponds to a truncated version of the Rezk completion for Segal spaces, and also to the stack completion of a prestack.
\end{abstract}

\maketitle

\section{Introduction}
\label{sec:introduction}

Of the branches of mathematics, category theory is one which perhaps fits the least comfortably into existing ``foundations of mathematics''.
This is true both at an informal level, and when trying to be completely formal using a computer proof assistant.
One problem is that naive category theory tends to run afoul of Russellian paradoxes and has to be reinterpreted using universe levels; we will not have much to say about this.
But another problem is that most of category theory is invariant under weaker notions of ``sameness'' than equality, such as isomorphism in a category or equivalence of categories, in a way which traditional foundations (such as set theory) fail to capture.
This problem becomes especially important when formalizing category theory in a computer proof assistant.

Our aim in this paper is to show that this problem can be ameliorated using the new \emph{Univalent Foundations} of mathematics, a.k.a.\ \emph{homotopy type theory}, proposed by V.~Voevodsky \cite{vv_uf}.
It builds on the existing system of dependent type theory \cite{martin-lof:bibliopolis, werner:thesis}, a logical system that is feasible for large-scale formalization of mathematics \cite{gonthier:feit-thompson} and also for internal categorical logic.
The distinctive feature of Univalent Foundations (UF) is its treatment of equality inspired by homotopy-theoretic semantics \cite{awodey-warren, arndt-kapulkin, warren:thesis, garner-van-den-berg:top-and-simp-models}.
Using this interpretation, Voevodsky has extended dependent type theory with an additional axiom, called the \emph{Univalence Axiom}, which was originally suggested by the model of the theory in the category of simplicial sets~\cite{klv:ssetmodel}, and should also be valid in other homotopical models such as categories of higher stacks.

The univalence axiom identifies \emph{identity} of types with \emph{equivalence} of types.
In particular, this implies that anything we can say about sets is automatically invariant under isomorphism, because isomorphism is identified with identity.
In other words, under the univalence axiom, the category of sets \emph{automatically} behaves ``categorically'', in that isomorphic objects cannot be distinguished.
Our goal in this paper is to extend this behavior to other categories, which requires a more careful analysis of the definition of ``category''.

If we ignore size issues, then in set-based mathematics, a category consists of a \emph{set} of objects and, for each pair $x,y$ of objects, a \emph{set} $\hom(x,y)$ of morphisms.
Under Univalent Foundations, a ``naive'' definition of category would simply mimic this with a \emph{type} of objects and \emph{types} of morphisms.
However, if we allowed these types to contain arbitrary higher homotopy, then we ought to impose higher coherence conditions on the associativity and unitality axioms, leading to some notion of $(\infty,1)$-category.
Eventually this should be done, but at present our goal is more modest.
We restrict ourselves to 1-categories, and therefore we restrict the hom-types $\hom(x,y)$ to be \emph{sets} in the sense of UF, i.e.\ types satisfying the principle UIP of ``uniqueness of identity proofs''.

More interesting is whether the type of objects should have any higher homotopy.
If we require it also to be a set, then we end up with a definition that behaves more like the traditional set-theoretic one.
Following Toby Bartels, we call this notion a \emph{strict category}.

However, a (usually) better option is to require a generalized version of the univalence axiom, identifying the \emph{identity type} $(x=_{\mathsf{Obj}} y)$ between two objects with the type $\mathsf{iso}(x,y)$ of \emph{isomorphisms} from $x$ to $y$.
(In particular, this implies that each type $(x=_{\mathsf{Obj}} y)$ is a set, and that therefore the type of objects is a \emph{1-type}, containing no higher homotopy above dimension 1.)
This seems to have been first suggested by Hofmann and Streicher~\cite{hs:gpd-typethy}, who also introduced a precursor of the univalence axiom under the name ``universe extensionality''.
We consider it to be the ``correct'' definition of \emph{category} in Univalent Foundations, since it automatically implies that anything we say about objects of a category is invariant under isomorphism.
For emphasis, we may call such a category a \emph{saturated} or \emph{univalent} category.

Most categories encountered in practice are saturated, at least in the presence of the univalence axiom.
Those which are not saturated, such as the category of $n$-types and homotopy classes of functions for $n\ge 1$, tend to behave much worse than the saturated ones.
Thus, in the non-saturated and non-strict case, we use instead the slightly derogatory word \emph{precategory}.

A good example of the difference between the three notions of category is provided by the statement ``every fully faithful and essentially surjective functor is an equivalence of categories'', which in classical set-based category theory is equivalent to the axiom of choice.
\begin{enumerate}
\item For strict categories, this is still equivalent to to the axiom of choice.
\item For precategories, there is no axiom of choice which can make it true.
\item For saturated categories, it is provable \emph{without} any axiom of choice.\label{item:satnoac}
\end{enumerate}
Saturated categories have the additional advantage that (as conjectured by Hofmann and Streicher~\cite{hs:gpd-typethy}) they are ``univalent as objects'' as well.
Specifically, just the way isomorphic objects \emph{in} a saturated category are equal, \emph{equivalent} saturated categories are themselves equal.

When interpreted in Voevodsky's simplicial set model, our precategories are similar to a truncated analogue of the \emph{Segal spaces} of Rezk~\cite[Sec.~14]{rezk01css}, while our saturated categories correspond to his \emph{complete Segal spaces}.
Strict categories correspond instead to (a weakened and truncated version of) \emph{Segal categories}.
It is known that Segal categories and complete Segal spaces are equivalent models for $(\infty,1)$-categories (see e.g.~\cite{bergner:infty-one}), so that in the simplicial set model, strict and saturated categories yield ``equivalent'' category theories---although as mentioned above, the saturated ones still have many advantages.

However, in the more general categorical semantics of a higher topos, a strict category corresponds to an internal category (in the traditional sense) in the corresponding 1-topos of sheaves, while a saturated category corresponds to a \emph{stack}.
Internal categories are \emph{not} equivalent to stacks (in fact, stacks form a localization of internal categories~\cite{jt:strong-stacks}), and it is well-known that stacks are generally a more appropriate sort of ``category'' relative to a topos.

Besides developing the basic theory of precategories and saturated categories, one of the main goals of this paper is to describe a universal way of ``saturating'' a precategory.
More precisely, we show that the obvious inclusion of saturated precategories into categories has a left adjoint, in the appropriate bicategorical sense.
More concretely, from any precategory $A$, we construct a saturated category $\widehat{A}$, with a universal functor
$A \to \widehat{A}$ (the unit of the adjunction).

With the connection to Rezk's complete Segal spaces in mind, we call the saturation of a precategory its \emph{Rezk completion}.
However, with higher topos semantics in mind, it could also reasonably be called the \emph{stack completion}: a strict category in the internal type theory of a higher topos corresponds to an internal category in the 1-topos of sheaves, and its Rezk completion is essentially its stack completion.
Our construction uses a Yoneda embedding as in~\cite{bunge:stacks-morita-internal} rather than a transfinite localization argument as in~\cite{jt:strong-stacks,rezk01css}, but it is also possible to mimic the latter more closely in type theory using ``higher inductive types''~\cite{ls:hits}.
A slightly expanded version of this paper, which includes this alternative proof, is included in~\cite[Chapter 9]{HoTTbook}.

The Rezk completion also sheds further light on the notion of equivalence of categories.
For instance, the functor $A \to \widehat{A}$ is always fully faithful and essentially surjective, hence a ``weak equivalence''.
It follows that a precategory is a saturated category exactly when it ``sees'' all fully faithful and essentially surjective functors as equivalences.
(The analogous facts for complete Segal spaces and stacks are well-known.)
In particular, the notion of saturated category is already inherent in the notion of ``fully faithful and essentially surjective functor''.

Finally, as mentioned above, one of the virtues of Univalent Foundations (and type theory more generally) is the feasibility of formalizing it in a computer proof assistant.
We have taken advantage of this by verifying large parts of the theory of precategories and saturated categories in the proof assistant \textsf{Coq}, building on Voevodsky's \emph{Foundations} library for UF~\cite{vv_foundations}.
In particular, the formalization includes the Rezk completion together with its universal property.
Our \textsf{Coq} files are attached to this arXiv submission.

\begin{rmk}
  Because saturated categories are the ``correct'' notion of category in UF, when working internally in UF we drop the adjective ``saturated'' and speak merely of \emph{categories}.
  The adjective is only necessary when comparing such categories to other ``external'' notions of category.
\end{rmk}

\subsection*{Outline of the paper}

In \S\ref{sec:background} we recall some definitions from Univalent Foundations.
Then in \S\S\ref{sec:cats}--\ref{sec:yoneda} we develop the basic theory of precategories and saturated categories informally, working entirely inside of Univalent Foundations.
We define functors, natural transformations, adjunctions, equivalences, and prove the Yoneda lemma.
We also show that equivalent categories are equal.
In \S\ref{sec:rezk} we construct the Rezk completion which, as described above, universally saturates any precategory.

Finally, \S\ref{sec:formalization} describes the content of our formalization, the organization of the source files, and the differences between informal presentation and its formal analog.
The actual \textsf{Coq} code is available as a supplement to this paper \cite{rezk_coq}.

\subsection*{Acknowledgements}
First and foremost, we would like to thank Vladimir Voevodsky for initiating the project of Univalent Foundations and for much assistance.
We are also very grateful to the organizers of the special year at the Institute for Advanced Study in 2012--2013, where much of this work was done.
The first- and the third-named author were supported by NSF grant DMS-1128155.
The second-named author was supported by NSF Grant DMS-1001191 (P.I. Steve Awodey).
Any opinions, findings, and conclusions or recommendations expressed in this material are those of the authors and do not necessarily reflect the views of the National Science Foundation.

The second-named author dedicates this work to his mother.

\section{Review of univalent foundations}
\label{sec:background}

Most of this paper is written in an informal style, with the intent of describing mathematics that could be formalized in Univalent Foundations, analogously to the way that traditional mathematics is discussed informally but is generally accepted to be formalizable in set theory.
We do not have space to give an introduction to UF here; instead we refer the reader to~\cite{pelayo-warren:univalent-foundations-paper}.
However, a brief reminder of the essential concepts may be helpful.

The basic objects are \emph{types}, which have \emph{elements}, with the basic judgment of elementhood denoted $a:A$.
There are the usual constructions on types such as dependent sums and dependent products, which we generally write about in English according to the propositions-as-types interpretation: we identify the activity of \emph{proving a theorem} with the activity of \emph{constructing a term in a type}.
For instance, a statement like ``for all $x:A$ we have $P(x)$'' indicates that we have an element of the type $\prd{x:A} P(x)$, while ``there exists an $x:A$ such that $P(x)$'' indicates $\sm{x:A} P(x)$.
Depending on context, we may also pronounce $\sm{x:A} P(x)$ as ``the type of $x:A$ such that $P(x)$'' and write it as $\setof{x:A | P(x)}$.

For $a,b:A$ there is an \emph{identity type} $a=b$ (or $a=_A b$ for emphasis), which in the homotopical semantics becomes a \emph{path type}.
It has the universal property that we may prove things about a general $p:a=b$ by restricting to the special case when $a$ and $b$ are the same and $p$ is ``reflexivity''.
We refer to this as \emph{path induction} or \emph{induction on identity}.
For instance, in this way we can show that if $(P(x))_{x:A}$ is a family of types indexed by $A$, and we have $p:a=_A b$ and $u:P(a)$, then we can \emph{transport} $u$ along $p$ to obtain an element $\trans{p}{u} : P(b)$.
Similarly, we can show that for any $f:A\to B$ and $p:x=_A y$, we have $f(p):f(x)=_B f(y)$, and we can compose paths (written $p\ct q$) and reverse paths (written $\opp{p}$).

The identity type of many types can be characterized up to equivalence (see below).
For instance, to say $(x,u) = (y,v)$ in $\sm{a:A} P(a)$ is equivalent to saying that $p:x =_A y$ and $\trans{p}{u} =_{P(y)} v$.
And to say $f=g$ in $\prd{a:A} P(a)$ is to say that $f(x)=g(x)$ for all $x:A$ (this is \emph{function extensionality}, which follows from the univalence axiom below).

A type $A$ is called a \emph{mere proposition} if for all $a,b:A$ we have $a=b$.
Homotopically, these are the spaces which, if nonempty, are contractible.
With this in mind, we call a type $A$ \emph{contractible} if it is a mere proposition and has an element $a:A$.
On the other hand, we call $A$ a \emph{set} if for all $a,b:A$, the type $a=b$ is a mere proposition.
Homotopically, these are the spaces which are equivalent to discrete ones.
More generally, $A$ is an \emph{$n$-type} if each $a=b$ is an $(n-1)$-type, with the 0-types being the sets, the $(-1)$-types the mere propositions, and the $(-2)$-types the contractible ones.
This exactly matches the traditional notion of \emph{homotopy $n$-type}.

A \emph{quasi-inverse} of a function $f:A\to B$ is a function $g:B\to A$ such that $\eta_x:x = g(f(x))$ for all $x:A$ and $\epsilon_y:f(g(y))=y$ for all $y:B$.
We say $f$ is an \emph{equivalence} if it has a quasi-inverse such that $f(\eta_x) \ct \epsilon_{f(x)} = \refl{f(x)}$ for all $x:A$.
In fact, if $f$ has a quasi-inverse, then it is an equivalence (by modifying $\epsilon$ or $\eta$); this is the usual way that we construct equivalences.
However, the type ``$f$ is an equivalence'' is better-behaved than ``$f$ has a quasi-inverse''; in particular it is a mere proposition.
We write $A\simeq B$ for the type $\sm{f:A\to B} \mathsf{isequiv}(f)$ of equivalences from $A$ to $B$.

In the formalization, we use an equivalent definition that $f:A\to B$ is an equivalence if for all $b:B$, its ``homotopy fiber'' $\sm{x:A} (f(x)=b)$ is contractible.
In some literature such functions are called ``weak equivalences'', but there is nothing weak about them, since in particular they have quasi-inverses.

The types in UF are stratified in a linearly ordered hierarchy of \emph{universes}, which are types whose elements are themselves types.
For most of the paper we avoid mentioning particular universes explicitly: we write simply ``\type'' to indicate \emph{some} universe.
This is called \emph{typical ambiguity}: universes are implicitly quantified over.
However, in \S\S\ref{sec:yoneda}--\ref{sec:rezk} we will be a little more careful.

All our universes are assumed to satisfy the univalence axiom, which says that for types $A,B:\type$ in some universe \type, the canonical map $(A=_\type B) \to (A\simeq B)$ is an equivalence.

We write $\set$ for the type $\sm{A:\type} \mathsf{isset}(A)$ of all sets (in some universe \type).
Technically, this is the type of pairs $(A,s)$ where $A$ is a type and $s$ inhabits the type ``$A$ is a set'', but since the latter type is a mere proposition, it is usually easy to ignore the distinction.
Similarly, we write $\prop \defeq \sm{A:\type} \mathsf{isprop}(A)$ for the type of all mere propositions.

One type forming operation we use in UF which is not as well-known in type theory is the \emph{propositional truncation} of a type $A$.
This is a type $\brck A$ that is a mere proposition, and has the universal property that whenever we want to prove a type $B$ (i.e.\ construct an element of $B$) assuming $\brck A$, and $B$ is a mere proposition, then we may assume $A$ instead of $\brck A$.
In the formalization, we define $\brck A$ with an impredicative encoding as
\[ \brck A \defeq \prd{P:\prop} (A\to P) \to P. \]
This depends for its correctness on an impredicativity axiom for mere propositions (every mere proposition is equivalent to one living in the smallest universe), and also lives in a higher universe level than $A$.
However, $\brck A$ can be constructed as a higher inductive type~\cite{ls:hits}, avoiding both of these issues.

In informal mathematical English, we use the adverb \emph{merely} to indicate the propositional truncation; thus for instance ``there merely exists an $x:A$ such that $P(x)$'' indicates $\bbrck{\sm{x:A} P(x)}$.
In contrast to the type-theoretic ``there exists'' which is strongly constructive, ``mere existence'' is more like the usual mathematical sort of ``there exists'' which does not imply that any particular choice of such an object has been specified.

The propositional truncation is actually the case $n=-1$ of a more general $n$-truncation operation, which makes any type $A$ into an $n$-type $\trunc n A$ in a universal way.
However, we will not have much need of the $n$-truncation for $n\ge 0$.

A function $f:A\to B$ between types is called a \emph{monomorphism} if for all $x,y:A$, the function $f:(x=y) \to (f(x)=f(y))$ is an equivalence.
If $A$ and $B$ are sets, then it is equivalent to say that for all $x,y:A$, if $f(x)=f(y)$, then $x=y$; so in this case we also say that $f$ is \emph{injective}.
Also if $A$ and $B$ are sets, we say that $f:A\to B$ is \emph{surjective} if for every $b:B$ there merely exists an $a:A$ such that $f(a)=b$.
If in this definition we leave out the adverb ``merely'', we call the resulting notion being \emph{split surjective}; in the absence of the axiom of choice the two are different.
(Type theorists are accustomed to use the phrase ``the axiom of choice'' for a provable statement which is really about commutation of dependent sums and products; in UF one can state an axiom of choice that behaves more like the familiar one in set theory.
However, we will not need any such axiom.)

\section{Categories and precategories}
\label{sec:cats}

We use a definition of category in which the arrows form a family of types indexed by the objects.
This matches the way hom-types are always used in category theory; for instance, we never even consider comparing two arrows unless we know their sources and targets agree.
Furthermore, it seems clear that for a theory of 1-categories, the hom-types should all be sets.
This leads us to the following.

\begin{defn}\label{ct:precategory}
  A \textbf{precategory} $A$ consists of the following.
  \begin{enumerate}
  \item A type $A_0$ of \emph{objects}.  We write $a:A$ for $a:A_0$.
  \item For each $a,b:A$, a set $\hom_A(a,b)$ of \emph{arrows} or \emph{morphisms}.
  \item For each $a:A$, a morphism $1_a:\hom_A(a,a)$.
  \item For each $a,b,c:A$, a function of type
    \[  \hom_A(b,c) \to \hom_A(a,b) \to \hom_A(a,c) \]
    denoted infix by $g\mapsto f\mapsto g\circ f$, or sometimes simply by $gf$.
  \item For each $a,b:A$ and $f:\hom_A(a,b)$, we have $\id f {1_b\circ f}$ and $\id f {f\circ 1_a}$.
  \item For each $a,b,c,d:A$ and $f:\hom_A(a,b)$, $g:\hom_A(b,c)$, $h:\hom_A(c,d)$, we have $\id {h\circ (g\circ f)}{(h\circ g)\circ f}$.
  \end{enumerate}
\end{defn}

The problem with the notion of precategory is that for objects $a,b:A$, we have two possibly-different notions of ``sameness''.
On the one hand, we have $\id[A_0]{a}{b}$.
But on the other hand, there is the standard categorical notion of \emph{isomorphism}.

\begin{defn}\label{ct:isomorphism}
  A morphism $f:\hom_A(a,b)$ is an \textbf{isomorphism} if there is a morphism $g:\hom_A(b,a)$ such that $\id{g\circ f}{1_a}$ and $\id{f\circ g}{1_b}$.
  We write $a\cong b$ for the type of such isomorphisms.
\end{defn}

\begin{lem}\label{ct:isoprop}
  For any $f:\hom_A(a,b)$, the type ``$f$ is an isomorphism'' is a mere proposition.
  Therefore, for any $a,b:A$ the type $a\cong b$ is a set.
\end{lem}
\begin{proof}
  Suppose given $g:\hom_A(b,a)$ and $\eta:(\id{1_a}{g\circ f})$ and $\epsilon:(\id{f\circ g}{1_b})$, and similarly $g'$, $\eta'$, and $\epsilon'$.
We must show $\id{(g,\eta,\epsilon)}{(g',\eta',\epsilon')}$.
  But since all hom-sets are sets, their identity types (in which $\eta$ and $\epsilon$ live) are mere propositions, so it suffices to show $\id g {g'}$.
  For this we have
  \[g' = 1_a\circ g' = (g\circ f)\circ g' = g\circ (f\circ g') = g\circ 1_b = g\]
  using $\eta$ and $\epsilon'$.
\end{proof}

If $f:a\cong b$, then we write $\inv f$ for its inverse, which by \autoref{ct:isoprop} is uniquely determined.

The only relationship between these two notions of sameness that we have in a precategory is the following.

\begin{lem}[\textsf{idtoiso}]\label{ct:idtoiso}
  If $A$ is a precategory and $a,b:A$, then
  \[(\id a b)\to (a \cong b).\]
\end{lem}
\begin{proof}
  By induction on identity, we may assume $a$ and $b$ are the same.
  But then we have $1_a:\hom_A(a,a)$, which is clearly an isomorphism.
\end{proof}

The intuitive similarity to the univalence axiom should be clear.
More precisely, we have the following:

\begin{eg}\label{ct:precatset}
  There is a precategory \uset, whose type of objects is \set, and with $\hom_{\uset}(A,B) \defeq (A\to B)$.
  The identity morphisms are identity functions and the composition is function composition.
  For this precategory, \autoref{ct:idtoiso} is equal to the restriction to sets of the canonical identity-to-equivalence map, which the univalence axiom asserts to be an equivalence.
\end{eg}

Thus, it is natural to make the following definition.

\begin{defn}\label{ct:category}
  A \textbf{category} is a precategory such that for all $a,b:A$, the function $\idtoiso_{a,b}$ from \autoref{ct:idtoiso} is an equivalence.
\end{defn}

In particular, in a category, if $a\cong b$, then $a=b$.

\begin{eg}\label{ct:eg:set}
  The univalence axiom implies immediately that \uset is a category.
  One can also show, using univalence, that any precategory of set-level structures such as groups, rings, topological spaces, etc.\ is a category; see for instance~\cite{dc:isoeq}.
\end{eg}

We also note the following.

\begin{lem}\label{ct:obj-1type}
  In a category, the type of objects is a 1-type.
\end{lem}
\begin{proof}
  It suffices to show that for any $a,b:A$, the type $\id a b$ is a set.
  But $\id a b$ is equivalent to $a \cong b$, which is a set.
\end{proof}

We write $\isotoid$ for the inverse $(a\cong b) \to (\id a b)$ of the map $\idtoiso$ from \autoref{ct:idtoiso}.
The following relationship between the two is important.

Recall the notion of \emph{transport} along a path, denoted $\trans p z$.
Additionally, if $p:\id a a'$ and $q:\id b b'$, then we write $(p,q)$ for the induced path of type $\id{(a,b)}{(a',b')}$.

\begin{lem}\label{ct:idtoiso-trans}
  For $p:\id a a'$ and $q:\id b b'$ and $f:\hom_A(a,b)$, we have
  \begin{equation}\label{ct:idtoisocompute}
    \id{\trans{(p,q)}{f}}
    {\idtoiso(q)\circ f \circ \inv{\idtoiso(p)}}
  \end{equation}
\end{lem}
\begin{proof}
  By induction, we may assume $p$ and $q$ are $\refl a$ and $\refl b$ respectively.
Then the left-hand side of~\eqref{ct:idtoisocompute} is simply $f$.
  But by definition, $\idtoiso(\refl a)$ is $1_a$, and $\idtoiso(\refl b)$ is $1_b$, so the right-hand side of~\eqref{ct:idtoisocompute} is $1_b\circ f\circ 1_a$, which is equal to $f$.
\end{proof}

Similarly, we can show
\begin{gather}
  \id{\idtoiso(\rev p)}{\inv {(\idtoiso(p))}}\\
  \id{\idtoiso(p\ct q)}{\idtoiso(q)\circ \idtoiso(p)}\\
  \id{\isotoid(f\circ e)}{\isotoid(e)\ct \isotoid(f)}
\end{gather}
and so on.

\begin{eg}\label{ct:orders}
  A precategory in which each set $\hom_A(a,b)$ is a mere proposition is equivalently a type $A_0$ equipped with a mere relation ``$\le$'' that is reflexive ($a\le a$) and transitive (if $a\le b$ and $b\le c$, then $a\le c$).
  We call this a \textbf{preorder}.

  In a preorder, a morphism $f\colon a\le b$ is an isomorphism just when there exists some proof $g\colon b\le a$.
  Thus, $a\cong b$ is the mere proposition that $a\le b$ and $b\le a$.
  Therefore, a preorder $A$ is a category just when (1) each type $a=b$ is a mere proposition, and (2) for any $a,b:A_0$ there exists a function $(a\cong b) \to (a=b)$.
  In other words, $A_0$ must be a set, and $\le$ must be antisymmetric (if $a\le b$ and $b\le a$, then $a=b$).
  We call this a \textbf{(partial) order} or a \textbf{poset}.
\end{eg}

\begin{eg}\label{ct:gaunt}
  If $A$ is a category, then $A_0$ is a set if and only if for any $a,b:A_0$, the type $a\cong b$ is a mere proposition.
  Classically, a category satisfies this condition if and only if it is equivalent to one in which every isomorphism is an identity morphism.
  A category of the latter sort is sometimes called \textbf{gaunt} (this term was introduced by Barwick and Schommer-Pries~\cite{bsp12infncats}).
\end{eg}

\begin{eg}\label{ct:discrete}
  For any 1-type $X$, there is a category with $X$ as its type of objects and with $\hom(x,y) \defeq (x=y)$.
  If $X$ is a set, we call this the \textbf{discrete} category on $X$.
  In general, we call it a \textbf{groupoid}.
\end{eg}

\begin{eg}\label{ct:fundgpd}
  For \emph{any} type $X$, there is a precategory with $X$ as its type of objects and with $\hom(x,y) \defeq \pizero{x=y}$, the 0-truncation of its identity type.

  We call this the \emph{fundamental pregroupoid} of $X$.
\end{eg}

\begin{eg}\label{ct:hoprecat}
  There is a precategory whose type of objects is \type and with $\hom(X,Y) \defeq \pizero{X\to Y}$.
  We call this the \emph{homotopy precategory of types}.
\end{eg}

\begin{rmk}\label{defn:strict}
  As suggested in the introduction, if a precategory has the property that its type $A_0$ of objects is a \emph{set}, we call it a \textbf{strict category}.
  We will not have much to say about strict categories in this paper, however.
\end{rmk}

\section{Functors and transformations}
\label{sec:transfors}

The following definitions are fairly obvious, and need no modification.

\begin{defn}\label{ct:functor}
  Let $A$ and $B$ be precategories.
  A \textbf{functor} $F:A\to B$ consists of
  \begin{enumerate}
  \item A function $F_0:A_0\to B_0$, generally also denoted $F$.
  \item For each $a,b:A$, a function $F_{a,b}:\hom_A(a,b) \to \hom_B(Fa,Fb)$, generally also denoted $F$.
  \item For each $a:A$, we have $\id{F(1_a)}{1_{Fa}}$.
  \item For each $a,b,c:A$ and $f:\hom_A(a,b)$ and $g:\hom_B(b,c)$, we have
    \[\id{F(g\circ f)}{Fg\circ Ff}.\]
  \end{enumerate}
\end{defn}

Note that by induction on identity, a functor also preserves \idtoiso.

\begin{defn}\label{ct:nattrans}
  For functors $F,G:A\to B$, a \textbf{natural transformation} $\gamma:F\to G$ consists of
  \begin{enumerate}
  \item For each $a:A$, a morphism $\gamma_a:\hom_B(Fa,Ga)$.
  \item For each $a,b:A$ and $f:\hom_A(a,b)$, we have $\id{Gf\circ \gamma_a}{\gamma_b\circ Ff}$.
  \end{enumerate}
\end{defn}

Since each type $\hom_B(Fa,Gb)$ is a set, its identity type is a mere proposition.
Thus, the naturality axiom is a mere proposition, so (invoking function extensionality) identity of natural transformations is determined by identity of their components.
In particular, for any $F$ and $G$, the type of natural transformations from $F$ to $G$ is again a set.

Similarly, identity of functors is determined by identity of the functions $A_0\to B_0$ and (transported along this) of the corresponding functions on hom-sets.

\begin{defn}\label{ct:functor-precat}
  For precategories $A,B$, there is a precategory $B^A$ defined by
  \begin{itemize}
  \item $(B^A)_0$ is the type of functors from $A$ to $B$.
  \item $\hom_{B^A}(F,G)$ is the type of natural transformations from $F$ to $G$.
  \end{itemize}
\end{defn}
\begin{proof}
  We define $(1_F)_a\defeq 1_{Fa}$.
  Naturality follows by the unit axioms of a precategory.
  For $\gamma:F\to G$ and $\delta:G\to H$, we define $(\delta\circ\gamma)_a\defeq \delta_a\circ \gamma_a$.
  Naturality follows by associativity.
  Similarly, the unit and associativity laws for $B^A$ follow from those for $B$.
\end{proof}

\begin{lem}\label{ct:natiso}
  A natural transformation $\gamma:F\to G$ is an isomorphism in $B^A$ if and only if each $\gamma_a$ is an isomorphism in $B$.
\end{lem}
\begin{proof}
  If $\gamma$ is an isomorphism, then we have $\delta:G\to F$ that is its inverse.
  By definition of composition in $B^A$, $(\delta\gamma)_a\jdeq \delta_a\gamma_a$ and similarly.
  Thus, $\id{\delta\gamma}{1_F}$ and $\id{\gamma\delta}{1_G}$ imply $\id{\delta_a\gamma_a}{1_{Fa}}$ and $\id{\gamma_a\delta_a}{1_{Ga}}$, so $\gamma_a$ is an isomorphism.

  Conversely, suppose each $\gamma_a$ is an isomorphism, with inverse called $\delta_a$, say.
We define a natural transformation $\delta:G\to F$ with components $\delta_a$; for the naturality axiom we have
  \[ Ff\circ \delta_a = \delta_b\circ \gamma_b\circ Ff \circ \delta_a = \delta_b\circ Gf\circ \gamma_a\circ \delta_a = \delta_b\circ Gf. \]
  Now since composition and identity of natural transformations is determined on their components, we have $\id{\gamma\delta}{1_G}$ and $\id{\delta\gamma}{1_F}$.
\end{proof}

The following result, due originally to Hofmann and Streicher~\cite{hs:gpd-typethy}, is fundamental.

\begin{thm}\label{ct:functor-cat}
  If $A$ is a precategory and $B$ is a category, then $B^A$ is a category.
\end{thm}
\begin{proof}
  Let $F,G:A\to B$; we must show that $\idtoiso:(\id{F}{G}) \to (F\cong G)$ is an equivalence.

  To give an inverse to it, suppose $\gamma:F\cong G$ is a natural isomorphism.
  Then for any $a:A$, we have an isomorphism $\gamma_a:Fa \cong Ga$, hence an identity $\isotoid(\gamma_a):\id{Fa}{Ga}$.
  By function extensionality, we have an identity $\bar{\gamma}:\id[(A_0\to B_0)]{F_0}{G_0}$.

  Now since the last two axioms of a functor are mere propositions, to show that $\id{F}{G}$ it will suffice to show that for any $a,b:A$, the functions
  \begin{align*}
    F_{a,b}&:\hom_A(a,b) \to \hom_B(Fa,Fb)\mathrlap{\qquad\text{and}}\\
    G_{a,b}&:\hom_A(a,b) \to \hom_B(Ga,Gb)
  \end{align*}
  become equal when transported along $\bar\gamma$.
  By computation for function extensionality, when applied to $a$, $\bar\gamma$ becomes equal to $\isotoid(\gamma_a)$.
  But by \autoref{ct:idtoiso-trans}, transporting $Ff:\hom_B(Fa,Fb)$ along $\isotoid(\gamma_a)$ and $\isotoid(\gamma_b)$ is equal to the composite $\gamma_b\circ Ff\circ \inv{(\gamma_a)}$, which by naturality of $\gamma$ is equal to $Gf$.

  This completes the definition of a function $(F\cong G) \to (\id F G)$.
  Now consider the composite
  \[ (\id F G) \to (F\cong G) \to (\id F G). \]
  Since hom-sets are sets, their identity types are mere propositions, so to show that two identities $p,q:\id F G$ are equal, it suffices to show that $\id[\id{F_0}{G_0}]{p}{q}$.
  But in the definition of $\bar\gamma$, if $\gamma$ were of the form $\idtoiso(p)$, then $\gamma_a$ would be equal to $\idtoiso(p_a)$ (this can easily be proved by induction on $p$).
  Thus, $\isotoid(\gamma_a)$ would be equal to $p_a$, and so by function extensionality we would have $\id{\bar\gamma}{p}$, which is what we need.

  Finally, consider the composite
  \[(F\cong G)\to  (\id F G) \to (F\cong G). \]
  Since identity of natural transformations can be tested componentwise, it suffices to show that for each $a$ we have $\id{\idtoiso(\bar\gamma)_a}{\gamma_a}$.
  But as observed above, we have $\id{\idtoiso(\bar\gamma)_a}{\idtoiso((\bar\gamma)_a)}$, while $\id{(\bar\gamma)_a}{\isotoid(\gamma_a)}$ by computation for function extensionality.
  Since $\isotoid$ and $\idtoiso$ are inverses, we have $\id{\idtoiso(\bar\gamma)_a}{\gamma_a}$.
\end{proof}

In particular, naturally isomorphic functors between categories (as opposed to precategories) are equal.

\begin{defn}
  For functors $F:A\to B$ and $G:B\to C$, their composite $G\circ F:A\to C$ is given by
  \begin{itemize}
  \item The composite $(G_0\circ F_0) : A_0 \to C_0$
  \item For each $a,b:A$, the composite
    \[(G_{Fa,Fb}\circ F_{a,b}):\hom_A(a,b) \to \hom_C(GFa,GFb).\]
  \end{itemize}
  It is easy to check the axioms.
\end{defn}

\begin{defn}\label{def:whisker}
  For functors $F:A\to B$ and $G,H:B\to C$ and a natural transformation $\gamma:G\to H$, the composite $(\gamma F):GF\to HF$ is given by
  \begin{itemize}
  \item For each $a:A$, the component $\gamma_{Fa}$.
  \end{itemize}
  Naturality is easy to check.
  Similarly, for $\gamma$ as above and $K:C\to D$, the composite $(K\gamma):KG\to KH$ is given by
  \begin{itemize}
  \item For each $b:B$, the component $K(\gamma_b)$.
  \end{itemize}
\end{defn}

\begin{lem}\label{ct:interchange}
  For functors $F,G:A\to B$ and $H,K:B\to C$ and natural transformations $\gamma:F\to G$ and $\delta:H\to K$, we have
  \[\id{(\delta G)(H\gamma)}{(K\gamma)(\delta F)}.\]
\end{lem}
\begin{proof}
  It suffices to check componentwise: at $a:A$ we have
  \begin{align*}
    ((\delta G)(H\gamma))_a
    &\jdeq (\delta G)_{a}(H\gamma)_a\\
    &\jdeq \delta_{Ga}\circ H(\gamma_a)\\
    &= K(\gamma_a) \circ \delta_{Fa} \hspace{2cm}\text{(by naturality of $\delta$)}\\
    &\jdeq (K \gamma)_a\circ (\delta F)_a\\
    &\jdeq ((K \gamma)(\delta F))_a.\qedhere
  \end{align*}
\end{proof}

Classically, one defines the ``horizontal composite'' of $\gamma:F\to G$ and $\delta:H\to K$ to be the common value of ${(\delta G)(H\gamma)}$ and ${(K\gamma)(\delta F)}$.
We will refrain from doing this, because while equal, these two transformations are not \emph{definitionally} equal.
This restraint also has the consequence that we can use the symbol $\circ$ (or juxtaposition) for all kinds of composition unambiguously: there is only one way to compose two natural transformations (as opposed to composing a natural transformation with a functor on either side).

\begin{lem}\label{ct:functor-assoc}
  Composition of functors is associative: $\id{H(GF)}{(HG)F}$.
\end{lem}
\begin{proof}
  Since composition of functions is associative, this follows immediately for the actions on objects and on homs.
  And since hom-sets are sets, the rest of the data is automatic.
\end{proof}

The equality in \autoref{ct:functor-assoc} is likewise not definitional.
(Composition of functions is definitionally associative, but the axioms that go into a functor must also be composed, and this breaks definitional associativity.)  For this reason, we need also to know about \emph{coherence} for associativity.

\begin{lem}\label{ct:pentagon}
  \autoref{ct:functor-assoc} is coherent, i.e.\ the following pentagon of equalities commutes:
  \[ \xymatrix{ & K(H(GF)) \ar[dl] \ar[dr]\\
    (KH)(GF) \ar[d] && K((HG)F) \ar[d]\\
    ((KH)G)F && (K(HG))F \ar[ll] }
  \]
\end{lem}
\begin{proof}
  As in \autoref{ct:functor-assoc}, this is evident for the actions on objects, and the rest is automatic.
\end{proof}

We will henceforth abuse notation by writing $H\circ G\circ F$ or $HGF$ for either $H(GF)$ or $(HG)F$, transporting along \autoref{ct:functor-assoc} whenever necessary.
We have a similar coherence result for units.

\begin{lem}\label{ct:units}
  For a functor $F:A\to B$, we have equalities $\id{(1_B\circ F)}{F}$ and $\id{(F\circ 1_A)}{F}$, such that given also $G:B\to C$, the following triangle of equalities commutes.
  \[ \xymatrix{
    G\circ (1_B \circ F) \ar[rr] \ar[dr] &&
    (G\circ 1_B)\circ F \ar[dl] \\
    & G \circ F.}
  \]
\end{lem}

\section{Adjunctions}
\label{sec:adjunctions}

We take as our definition of adjunction the purely diagrammatic one in terms of a unit and counit natural transformation.

\begin{defn}\label{def:adjoint}
  A functor $F:A\to B$ is a \textbf{left adjoint} if there exists
  \begin{itemize}
  \item A functor $G:B\to A$.
  \item A natural transformation $\eta:1_A \to GF$.
  \item A natural transformation $\epsilon:FG\to 1_B$.
  \item $\id{(\epsilon F)(F\eta)}{1_F}$.
  \item $\id{(G\epsilon)(\eta G)}{1_G}$.
  \end{itemize}
\end{defn}

\begin{lem}\label{ct:adjprop}
  If $A$ is a category (but $B$ may be only a precategory), then the type ``$F$ is a left adjoint'' is a mere proposition.
\end{lem}
\begin{proof}
  Suppose given $(G,\eta,\epsilon)$ with the triangle identities and also $(G',\eta',\epsilon')$.
  Define $\gamma:G\to G'$ to be $(G'\epsilon)(\eta' G)$, and $\delta:G'\to G$ to be $(G\epsilon')(\eta G')$.
  Then
  \begin{align*}
    \delta\gamma &=
    (G\epsilon')(\eta G')(G'\epsilon)(\eta'G)\\
    &= (G\epsilon')(G F G'\epsilon)(\eta G' F G)(\eta'G)\\
    &= (G\epsilon)(G\epsilon'FG)(G F \eta' G)(\eta G)\\
    &= (G\epsilon)(\eta G)\\
    &= 1_G
  \end{align*}
  using \autoref{ct:interchange} and the triangle identities.
  Similarly, we show $\id{\gamma\delta}{1_{G'}}$, so $\gamma$ is a natural isomorphism $G\cong G'$.
  By \autoref{ct:functor-cat}, we have an identity $\id G {G'}$.

  Now we need to know that when $\eta$ and $\epsilon$ are transported along this identity, they become equal to $\eta'$ and $\epsilon'$.
  By \autoref{ct:idtoiso-trans}, this transport is given by composing with $\gamma$ or $\delta$ as appropriate.
  For $\eta$, this yields
  \begin{equation*}
    (G'\epsilon F)(\eta'GF)\eta
    = (G'\epsilon F)(G'F\eta)\eta'
    = \eta'
  \end{equation*}
  using \autoref{ct:interchange} and the triangle identity.
  The case of $\epsilon$ is similar.
  Finally, the triangle identities transport correctly automatically, since hom-sets are sets.
\end{proof}

In \S\ref{sec:yoneda} we will mention another way to prove \autoref{ct:adjprop}.

\section{Equivalences}
\label{sec:equivalences}

It is usual to define an equivalence of categories to be a functor $F:A\to B$ for which there exists a functor $G:B\to A$ and natural isomorphisms $F\circ G \cong 1_B$ and $G\circ F \cong 1_A$.
However, because of the ``proof-relevant'' or ``constructive'' nature of ``there exists'' (dependent sum types) in UF, this definition does not produce a well-behaved \emph{type of equivalences} between two categories.
The solution is not surprising to a category theorist: whenever equivalences are ill-behaved, it usually suffices to consider \emph{adjoint} equivalences instead.
(This is exactly the same problem and solution as is encountered in the definition of equivalence of \emph{types} in UF.)

\begin{defn}
  A functor $F:A\to B$ is an \textbf{equivalence of (pre)categories} if it is a left adjoint for which $\eta$ and $\epsilon$ are isomorphisms.
  We write $A\simeq B$ for the type of equivalences of categories from $A$ to $B$.
\end{defn}

By \autoref{ct:adjprop} and \autoref{ct:isoprop}, if $A$ is a category, then the type ``$F:A\to B$ is an equivalence of precategories'' is a mere proposition.

\begin{lem}\label{ct:adjointification}
  If for $F:A\to B$ there exists $G:B\to A$ and isomorphisms $GF\cong 1_A$ and $FG\cong 1_B$, then $F$ is an equivalence of precategories.
\end{lem}
\begin{proof}
  We can repeat the standard proof that any equivalence of categories gives rise to an adjoint equivalence.
  First note that for any $a:A$ we have
  \begin{equation}
    \eta_{GFa} = GF(\eta_a).\label{eq:gfeta}
  \end{equation}
  This follows by cancelling $\eta_a$ in the naturality condition $\eta_{GFa} \circ \eta_a = GF(\eta_a) \circ \eta_a$.

  Now, given $G$ and $\eta:FG\cong 1_B$ and $\epsilon : 1_A \cong GF$, we define $\epsilon'$ by
  \[ \epsilon'_b \defeq
  \epsilon_b \circ
  F(\eta_{Gb})^{-1} \circ
  (\epsilon_{FGb})^{-1}.
  \]
  This is evidently a natural isomorphism.
  Then we have
  \begin{align*}
    \epsilon'_{Fa} \circ F\eta_{a}
    &= \epsilon_{Fa} \circ F(\eta_{GFa})^{-1} \circ (\epsilon_{FGFa})^{-1} \circ F\eta_a\\
    &= \epsilon_{Fa} \circ FGF(\eta_{a})^{-1} \circ (\epsilon_{FGFa})^{-1} \circ F\eta_a\\
    &= \epsilon_{Fa} \circ (\epsilon_{Fa})^{-1} \circ F(\eta_{a})^{-1} \circ F\eta_a\\
    &= 1_{Fa}.
  \end{align*}
  using~\eqref{eq:gfeta} and naturality of $\epsilon$.
  For the other identity $G(\epsilon_b) \circ \eta_{Gb} = 1_{Gb}$, it suffices to show $G(\epsilon_b) \circ GFG(\epsilon_b) = \eta_{Gb}^{-1} \circ GFG(\epsilon_b)$.
  But we have
  \begin{align*}
    \eta_{Gb}^{-1} \circ GFG(\epsilon'_b)
    &= G(\epsilon'_b) \circ \eta_{GFGb}^{-1}\\
    &= G(\epsilon'_b) \circ GF(\eta_{Gb})^{-1}\\
    &= G(\epsilon'_b) \circ G(\epsilon'_{FGb})\\
    &= G(\epsilon'_b) \circ GFG(\epsilon'_b)
  \end{align*}
  using naturality of $\eta$,~\eqref{eq:gfeta}, the previous identity, and naturality of $\epsilon'$.
\end{proof}

We now investigate some alternative definitions of equivalences of categories.

\begin{defn}
  We say a functor $F:A\to B$ is \textbf{faithful} if for all $a,b:A$, the function
  \[F_{a,b}:\hom_A(a,b) \to \hom_B(Fa,Fb)\]
  is injective, and \textbf{full} if for all $a,b:A$ this function is surjective.
  If it is both (hence each $F_{a,b}$ is an equivalence) we say $F$ is \textbf{fully faithful}.
\end{defn}

\begin{defn}
  We say a functor $F:A\to B$ is \textbf{split essentially surjective} if for all $b:B$ there exists an $a:A$ such that $Fa\cong b$.
\end{defn}

The reason for the adjective \emph{split} is that because of the strong type-theoretic meaning of ``there exists'', such a functor comes with a function assigning a specified $a$ for every $b$.
This has the following advantage.

\begin{lem}\label{ct:ffeso}
  For any precategories $A$ and $B$ and functor $F:A\to B$, the following types are equivalent.
  \begin{enumerate}
  \item $F$ is an equivalence of precategories.\label{item:ct:ffeso1}
  \item $F$ is fully faithful and split essentially surjective.\label{item:ct:ffeso2}
  \end{enumerate}
\end{lem}
\begin{proof}
  Suppose $F$ is an equivalence of precategories, with $G,\eta,\epsilon$ specified.
  Then we have the function
  \begin{equation*}
    \begin{array}{rcl}
      \hom_B(Fa,Fb) &\to& \hom_A(a,b)\\
      g &\mapsto& \inv{\eta_b}\circ G(g)\circ \eta_a.
    \end{array}
  \end{equation*}
  For $f:\hom_A(a,b)$, we have
  \[ \inv{\eta_{b}}\circ G(F(f))\circ \eta_{a}  =
  \inv{\eta_{b}} \circ \eta_{b} \circ f=
  f
  \]
  while for $g:\hom_B(Fa,Fb)$ we have
  \begin{align*}
    F(\inv{\eta_b} \circ G(g)\circ\eta_a)
    &= F(\inv{\eta_b})\circ F(G(g))\circ F(\eta_a)\\
    &= \epsilon_{Fb}\circ F(G(g))\circ F(\eta_a)\\
    &= g\circ\epsilon_{Fa}\circ F(\eta_a)\\
    &= g
  \end{align*}
  using naturality of $\epsilon$, and the triangle identities twice.
  Thus, $F_{a,b}$ is an equivalence, so $F$ is fully faithful.
  Finally, for any $b:B$, we have $Gb:A$ and $\epsilon_b:FGb\cong b$.

  On the other hand, suppose $F$ is fully faithful and split essentially surjective.
  Define $G_0:B_0\to A_0$ by sending $b:B$ to the $a:A$ given by the specified essential splitting, and write $\epsilon_b$ for the likewise specified isomorphism $FGb\cong b$.

  Now for any $g:\hom_B(b,b')$, define $G(g):\hom_A(Gb,Gb')$ to be the unique morphism such that $\id{F(G(g))}{\inv{(\epsilon_{b'})}\circ g \circ \epsilon_b }$ (which exists since $F$ is fully faithful).
  Finally, for $a:A$ define $\eta_a:\hom_A(a,GFa)$ to be the unique morphism such that $\id{F\eta_a}{\inv{\epsilon_{Fa}}}$.
  It is easy to verify that $G$ is a functor and that $(G,\eta,\epsilon)$ exhibit $F$ as an equivalence of precategories.

  Now consider the composite~\ref{item:ct:ffeso1}$\to$\ref{item:ct:ffeso2}$\to$\ref{item:ct:ffeso1}.
  We clearly recover the same function $G_0:B_0 \to A_0$.
  For the action of $G$ on hom-sets, we must show that for $g:\hom_B(b,b')$, $G(g)$ is the (necessarily unique) morphism such that $F(G(g)) = \inv{(\epsilon_{b'})}\circ g \circ \epsilon_b$.
  But this equation holds by the assumed naturality of $\epsilon$.
  We also clearly recover $\epsilon$, while $\eta$ is uniquely characterized by $\id{F\eta_a}{\inv{\epsilon_{Fa}}}$ (which is one of the triangle identities assumed to hold in the structure of an equivalence of precategories).
  Thus, this composite is equal to the identity.

  Finally, consider the other composite~\ref{item:ct:ffeso2}$\to$\ref{item:ct:ffeso1}$\to$\ref{item:ct:ffeso2}.
  Since being fully faithful is a mere proposition, it suffices to observe that we recover, for each $b:B$, the same $a:A$ and isomorphism $F a \cong b$.
  But this is clear, since we used this function and isomorphism to define $G_0$ and $\epsilon$ in~\ref{item:ct:ffeso1}, which in turn are precisely what we used to recover~\ref{item:ct:ffeso2} again.
  Thus, the composites in both directions are equal to identities, hence we have an equivalence \eqv{\ref{item:ct:ffeso1}}{\ref{item:ct:ffeso2}}.
\end{proof}

However, if $B$ is not a category, then neither type in \autoref{ct:ffeso} may necessarily be a mere proposition.
Moreover, classically, one usually defines ``essentially surjective'' without specifying the witnesses in a determinate way.
In UF, the appropriate version of this definition is the following.

\begin{defn}
  A functor $F:A\to B$ is \textbf{essentially surjective} if for all $b:B$, there \emph{merely} exists an $a:A$ such that $Fa\cong b$.
  We say $F$ is a \textbf{weak equivalence} if it is fully faithful and essentially surjective.
\end{defn}

Being a weak equivalence is \emph{always} a mere proposition, since a function being an equivalence of types is such, and the propositional truncation is so by definition.
For categories, however, there is no difference between equivalences and weak ones.

\begin{lem}\label{ct:catweq}
  If $F:A\to B$ is fully faithful and $A$ is a category, then for any $b:B$ the type $\sm{a:A} (Fa\cong b)$ is a mere proposition.
  Hence if $A$ and $B$ are categories, then the types ``$F$ is an equivalence'' and ``$F$ is a weak equivalence'' are equivalent (and mere propositions).
\end{lem}
\begin{proof}
  Suppose given $(a,f)$ and $(a',f')$ in $\sm{a:A} (Fa\cong b)$.
  Then $\inv{f'}\circ f$ is an isomorphism $Fa \cong Fa'$.
  Since $F$ is fully faithful, we have $g:a\cong a'$ with $Fg = \inv{f'}\circ f$.
  And since $A$ is a category, we have $p:a=a'$ with $\idtoiso(p)=g$.
  Now $Fg = \inv{f'}\circ f$ implies $\trans{(\map{(F_0)}{p})}{f} = f'$, hence (by the characterization of equalities in dependent sums) $(a,f)=(a',f')$.

  Thus, for fully faithful functors whose domain is a category, essential surjectivity is equivalent to split essential surjectivity, and so being a weak equivalence is equivalent to being an equivalence.
\end{proof}

This is an important advantage of our category theory over set-based approaches.
As remarked in the introduction, with a purely set-based definition of category, the statement ``every fully faithful and essentially surjective functor is an equivalence of categories'' is equivalent to the axiom of choice (in the appropriate sense of UF).
Here we have it for free, as a category-theoretic version of the function comprehension principle.
We will see in \S\ref{sec:rezk} that this property moreover characterizes categories among precategories.

On the other hand, the following characterization of equivalences of categories is perhaps even more useful.

\begin{defn}\label{ct:isocat}
  A functor $F:A\to B$ is an \textbf{isomorphism of (pre)categories} if $F$ is fully faithful and $F_0:A_0\to B_0$ is an equivalence of types.
\end{defn}

Note that being an isomorphism of precategories is always a mere proposition.
Let $A\cong B$ denote the type of isomorphisms of (pre)categories from $A$ to $B$.

\begin{lem}\label{ct:isoprecat}
  For precategories $A$ and $B$ and $F:A\to B$, the following types are equivalent.
  \begin{enumerate}
  \item $F$ is an isomorphism of precategories.\label{item:ct:ipc1}
  \item There exist $G:B\to A$ and $\eta:1_A = GF$ and $\epsilon:FG=1_B$ such that\label{item:ct:ipc2}
    \begin{equation}
      \map{(F\circ -)}{\eta} = \map{(-\circ F)}{\opp\epsilon}.\label{eq:ct:isoprecattri}
    \end{equation}
  \item There merely exist $G:B\to A$ and $\eta:1_A = GF$ and $\epsilon:FG=1_B$.\label{item:ct:ipc3}
  \end{enumerate}
\end{lem}

In~\eqref{eq:ct:isoprecattri}, $\map{(F\circ -)}{\eta}$ denotes application of the function $(F\circ -)$ (which goes from functors $A\to A$ to functors $A\to B$) to the equality $\eta$, and similarly for $\map{(-\circ F)}{\opp\epsilon}$.
Note that if $B_0$ is not a 1-type, then~\eqref{eq:ct:isoprecattri} may not be a mere proposition.

\begin{proof}
  First note that since hom-sets are sets, equalities between equalities of functors are uniquely determined by their object-parts.
  Thus, by function extensionality,~\eqref{eq:ct:isoprecattri} is equivalent to
  \begin{equation}
    \map{(F_0)}{\eta_0}_a = \opp{(\epsilon_0)}_{F_0 a}.\label{eq:ct:ipctri}
  \end{equation}
  for all $a:A_0$.
  Note that this is precisely the coherence condition for $G_0$, $\eta_0$, and $\epsilon_0$ to be a proof that $F_0$ is an equivalence of types.

  Now suppose~\ref{item:ct:ipc1}.
  Let $G_0:B_0 \to A_0$ be the inverse of $F_0$, with $\eta_0: \idfunc[A_0] = G_0 F_0$ and $\epsilon_0:F_0G_0 = \idfunc[B_0]$ satisfying the triangle identity, which is precisely~\eqref{eq:ct:ipctri}.
  Now define $G_{b,b'}:\hom_B(b,b') \to \hom_A(G_0b,G_0b')$ by
  \[ G_{b,b'}(g) \defeq
  \inv{(F_{G_0b,G_0b'})}\Big(\idtoiso(\opp{(\epsilon_0)}_{b'}) \circ g \circ \idtoiso((\epsilon_0)_b)\Big)
  \]
  (using the assumption that $F$ is fully faithful).
  Since \idtoiso takes opposites to inverses and concatenation to composition, and $F$ is a functor, it follows that $G$ is a functor.

  By definition, we have $(GF)_0 \jdeq G_0 F_0$, which is equal to $\idfunc[A_0]$ by $\eta_0$.
  To obtain $1_A = GF$, we need to show that when transported along $\eta_0$, the identity function of $\hom_A(a,a')$ becomes equal to the composite $G_{Fa,Fa'} \circ F_{a,a'}$.
  In other words, for any $f:\hom_A(a,a')$ we must have
  \begin{multline*}
    \idtoiso((\eta_0)_{a'}) \circ f \circ \idtoiso(\opp{(\eta_0)}_a)\\
    = \inv{(F_{GFa,GFa'})}\Big(\idtoiso(\opp{(\epsilon_0)}_{Fa'})
    \circ F_{a,a'}(f) \circ \idtoiso((\epsilon_0)_{Fa})\Big).
  \end{multline*}
  But this is equivalent to
  \begin{multline*}
    (F_{GFa,GFa'})\Big(\idtoiso((\eta_0)_{a'}) \circ f \circ \idtoiso(\opp{(\eta_0)}_a)\Big)\\
    = \idtoiso(\opp{(\epsilon_0)}_{Fa'})
    \circ F_{a,a'}(f) \circ \idtoiso((\epsilon_0)_{Fa}).
  \end{multline*}
  which follows from functoriality of $F$, the fact that $F$ preserves \idtoiso, and~\eqref{eq:ct:ipctri}.
  Thus we have $\eta:1_A = GF$.

  On the other side, we have $(FG)_0\jdeq F_0 G_0$, which is equal to $\idfunc[B_0]$ by $\epsilon_0$.
  To obtain $FG=1_B$, we need to show that when transported along $\epsilon_0$, the identity function of $\hom_B(b,b')$ becomes equal to the composite $F_{Gb,Gb'} \circ G_{b,b'}$.
  That is, for any $g:\hom_B(b,b')$ we must have
  \begin{multline*}
    F_{Gb,Gb'}\Big(\inv{(F_{Gb,Gb'})}\Big(\idtoiso(\opp{(\epsilon_0)}_{b'}) \circ g \circ \idtoiso((\epsilon_0)_b)\Big)\Big)\\
    = \idtoiso((\opp{\epsilon_0})_{b'}) \circ g \circ \idtoiso((\epsilon_0)_b).
  \end{multline*}
  But this is just the fact that $\inv{(F_{Gb,Gb'})}$ is the inverse of $F_{Gb,Gb'}$.
  And we have remarked that~\eqref{eq:ct:isoprecattri} is equivalent to~\eqref{eq:ct:ipctri}, so~\ref{item:ct:ipc2} holds.

  Conversely, suppose given~\ref{item:ct:ipc2}; then the object-parts of $G$, $\eta$, and $\epsilon$ together with~\eqref{eq:ct:ipctri} show that $F_0$ is an equivalence of types.
  And for $a,a':A_0$, we define $\overline{G}_{a,a'}: \hom_B(Fa,Fa') \to \hom_A(a,a')$ by
  \begin{equation}
    \overline{G}_{a,a'}(g) \defeq \idtoiso(\opp{\eta})_{a'} \circ G(g) \circ \idtoiso(\eta)_a.\label{eq:ct:gbar}
  \end{equation}
  By naturality of $\idtoiso(\eta)$, for any $f:\hom_A(a,a')$ we have
  \begin{align*}
    \overline{G}_{a,a'}(F_{a,a'}(f))
    &= \idtoiso(\opp{\eta})_{a'} \circ G(F(f)) \circ \idtoiso(\eta)_a\\
    &= \idtoiso(\opp{\eta})_{a'} \circ \idtoiso(\eta)_{a'} \circ f \\
    &= f.
  \end{align*}
  On the other hand, for $g:\hom_B(Fa,Fa')$ we have
  \begin{align*}
    F_{a,a'}(\overline{G}_{a,a'}(g))
    &= F(\idtoiso(\opp{\eta})_{a'}) \circ F(G(g)) \circ F(\idtoiso(\eta)_a)\\
    &= \idtoiso(\epsilon)_{Fa'}
    \circ F(G(g))
    \circ \idtoiso(\opp{\epsilon})_{Fa}\\
    &= \idtoiso(\epsilon)_{Fa'}
    \circ \idtoiso(\opp{\epsilon})_{Fa'}
    \circ g\\
    &= g.
  \end{align*}
  (There are lemmas needed here regarding the compatibility between \idtoiso and whiskering, which we leave to the reader to state and prove.)
  Thus, $F_{a,a'}$ is an equivalence, so $F$ is fully faithful; i.e.~\ref{item:ct:ipc1} holds.

  Now the composite~\ref{item:ct:ipc1}$\to$\ref{item:ct:ipc2}$\to$\ref{item:ct:ipc1} is equal to the identity since~\ref{item:ct:ipc1} is a mere proposition.
  On the other side, tracing through the above constructions we see that the composite~\ref{item:ct:ipc2}$\to$\ref{item:ct:ipc1}$\to$\ref{item:ct:ipc2} essentially preserves the object-parts $G_0$, $\eta_0$, $\epsilon_0$, and the object-part of~\eqref{eq:ct:isoprecattri}.
  And in the latter three cases, the object-part is all there is, since hom-sets are sets.

  Thus, it suffices to show that we recover the action of $G$ on hom-sets.
  In other words, we must show that if $g:\hom_B(b,b')$, then
  \[ G_{b,b'}(g) =
  \overline{G}_{G_0b,G_0b'}\Big(\idtoiso(\opp{(\epsilon_0)}_{b'}) \circ g \circ \idtoiso((\epsilon_0)_b)\Big)
  \]
  where $\overline{G}$ is defined by~\eqref{eq:ct:gbar}.
  However, this follows from functoriality of $G$ and the \emph{other} triangle identity, which is equivalent to~\eqref{eq:ct:ipctri}.

  Now since~\ref{item:ct:ipc1} is a mere proposition, so is~\ref{item:ct:ipc2}, so it suffices to show they are co-inhabited with~\ref{item:ct:ipc3}.
  Of course,~\ref{item:ct:ipc2}$\to$\ref{item:ct:ipc3}, so let us assume~\ref{item:ct:ipc3}.
  Since~\ref{item:ct:ipc1} is a mere proposition, we may assume given $G$, $\eta$, and $\epsilon$.
  Then $G_0$ along with $\eta$ and $\epsilon$ imply that $F_0$ is an equivalence.
  Moreover, we also have natural isomorphisms $\idtoiso(\eta):1_A\cong GF$ and $\idtoiso(\epsilon):FG\cong 1_B$, so by \autoref{ct:adjointification}, $F$ is an equivalence of precategories, and in particular fully faithful.
\end{proof}

From \autoref{ct:isoprecat}\ref{item:ct:ipc2} and $\idtoiso$ in functor categories, we conclude immediately that any isomorphism of precategories is an equivalence.
For precategories, the converse can fail.

\begin{eg}\label{ct:chaotic}
  Let $X$ be a type and $x_0:X$ an element, and let $X_{\mathrm{ch}}$ denote the \emph{chaotic} or \emph{indiscrete} precategory on $X$.
  By definition, we have $(X_{\mathrm{ch}})_0\defeq X$, and $\hom_{X_{\mathrm{ch}}}(x,x') = 1$ for all $x,x'$.
  Then the unique functor $X_{\mathrm{ch}}\to 1$ is an equivalence of precategories, but not an isomorphism unless $X$ is contractible.

  This example also shows that a precategory can be equivalent to a category without itself being a category.
  Of course, if a precategory is \emph{isomorphic} to a category, then it must itself be a category.
\end{eg}

However, for categories, the notions of equivalence and isomorphism coincide.

\begin{lem}\label{ct:eqv-levelwise}
  For categories $A$ and $B$, a functor $F:A\to B$ is an equivalence of categories if and only if it is an isomorphism of categories.
\end{lem}
\begin{proof}
  Since both are mere properties, it suffices to show they are co-inhabited.
  So first suppose $F$ is an equivalence of categories, with $(G,\eta,\epsilon)$ given.
  We have already seen that $F$ is fully faithful.
  By \autoref{ct:functor-cat}, the natural isomorphisms $\eta$ and $\epsilon$ yield identities $\id{1_A}{GF}$ and $\id{FG}{1_B}$, hence in particular identities $\id{\idfunc[A]}{G_0\circ F_0}$ and $\id{F_0\circ G_0}{\idfunc[B]}$.
Thus, $F_0$ is an equivalence of types.

  Conversely, suppose $F$ is fully faithful and $F_0$ is an equivalence of types, with inverse $G_0$, say.
  Then for each $b:B$ we have $G_0 b:A$ and an identity $\id{FGb}{b}$, hence an isomorphism $FGb\cong b$.
  Thus, by \autoref{ct:ffeso}, $F$ is an equivalence of categories.
\end{proof}

Of course, there is yet a third notion of sameness for (pre)categories: equality.
However, the univalence axiom implies that it coincides with isomorphism.

\begin{lem}\label{ct:cat-eq-iso}
  If $A$ and $B$ are precategories, then the function
  \[(\id A B) \to (A\cong B)\]
  (defined by induction from the identity functor) is an equivalence of types.
\end{lem}
\begin{proof}
  As usual for dependent sum types, to give an element of $\id A B$ is equivalent to giving
  \begin{itemize}
  \item an identity $P_0:\id{A_0}{B_0}$,
  \item for each $a,b:A_0$, an identity
    \[P_{a,b}:\id{\hom_A(a,b)}{\hom_B(\trans {P_0} a,\trans {P_0} b)},\]
  \item identities $\id{\trans {(P_{a,a})} {1_a}}{1_{\trans {P_0} a}}$ and $\id{\trans {(P_{a,c})} {gf}}{\trans {(P_{b,c})} g \circ \trans {(P_{a,b})} f}$.
  \end{itemize}
  (Again, we use the fact that the identity types of hom-sets are mere propositions.)
  However, by univalence, this is equivalent to giving
  \begin{itemize}
  \item an equivalence of types $F_0:\eqv{A_0}{B_0}$,
  \item for each $a,b:A_0$, an equivalence of types
    \[F_{a,b}:\eqv{\hom_A(a,b)}{\hom_B(F_0 (a),F_0 (b))},\]
  \item and identities $\id{F_{a,a}(1_a)}{1_{F_0 (a)}}$ and $\id{F_{a,c}(gf)}{F_{b,c} (g)\circ F_{a,b} (f)}$.
  \end{itemize}
  But this consists exactly of a functor $F:A\to B$ that is an isomorphism of categories.
  And by induction on identity, this equivalence $\eqv{(\id A B)}{(A\cong B)}$ is equal to the function obtained by induction.
\end{proof}

Thus, for categories, equality also coincides with equivalence.
We can interpret this as follows: define a ``pre-2-category'' to have a type of objects equipped with hom-precategories, composition functors, and so on.
Then categories, functors, and natural transformations form a pre-2-category whose hom-precategories are categories (this is \autoref{ct:functor-cat}), and \autoref{ct:cat-eq-iso} is a categorified version of the saturation property.
It is consistent to use the word \emph{2-category} for a pre-2-category satisfying both of these conditions.

The following corollary was conjectured by Hofmann and Streicher\cite{hs:gpd-typethy}.

\begin{thm}\label{ct:cat-2cat}
  If $A$ and $B$ are categories, then the function
  \[(\id A B) \to (A\simeq B)\]
  (defined by induction from the identity functor) is an equivalence of types.
\end{thm}
\begin{proof}
  By \autoref{ct:cat-eq-iso} and \autoref{ct:eqv-levelwise}.
\end{proof}

As a consequence, the type of categories is a 2-type.
For since $A\simeq B$ is a subtype of the type of functors from $A$ to $B$, which are the objects of a category, it is a 1-type; hence the identity types $\id A B$ are also 1-types.

\section{The Yoneda lemma}
\label{sec:yoneda}

In this section we fix a particular universe \type, and write \set for the type of sets in that universe and \uset for the category whose objects are sets in that universe and whose morphisms are functions between them.
Of course, \set and \uset do not themselves lie in the universe \type, but rather in some higher universe.

Define a precategory to be \emph{locally small} if its hom-sets lie in our fixed universe \type.
We now show that every locally small precategory has a \uset-valued hom-functor.
First we need to define opposites and products of (pre)categories.

\begin{defn}
  For a precategory $A$, its \textbf{opposite} $A\op$ is a precategory with the same type of objects, with $\hom_{A\op}(a,b) \defeq \hom_A(b,a)$, and with identities and composition inherited from $A$.
\end{defn}

\begin{defn}
  For precategories $A$ and $B$, their \textbf{product} $A\times B$ is a precategory with $(A\times B)_0 \defeq A_0 \times B_0$ and
  \[\hom_{A\times B}((a,b),(a',b')) \defeq \hom_A(a,a') \times \hom_B(b,b').\]
  Identities are defined by $1_{(a,b)}\defeq (1_a,1_b)$ and composition by $(g,g')(f,f') \defeq ((gf),(g'f'))$.
\end{defn}

\begin{lem}\label{ct:functorexpadj}
  For precategories $A,B,C$, the following types are equivalent.
  \begin{enumerate}
  \item Functors $A\times B\to C$.
  \item Functors $A\to C^B$.
  \end{enumerate}
\end{lem}
\begin{proof}
  Given $F:A\times B\to C$, for any $a:A$ we obviously have a functor $F_a : B\to C$.
  This gives a function $A_0 \to (C^B)_0$.
  Next, for any $f:\hom_A(a,a')$, we have for any $b:B$ the morphism $F_{(a,b),(a',b)}(f,1_b):F_a(b) \to F_{a'}(b)$.
  These are the components of a natural transformation $F_a \to F_{a'}$.
  Functoriality in $a$ is easy to check, so we have a functor $\widehat{F}:A\to C^B$.

  Conversely, suppose given $G:A\to C^B$.
  Then for any $a:A$ and $b:B$ we have the object $G(a)(b):C$, giving a function $A_0 \times B_0 \to C_0$.
  And for $f:\hom_A(a,a')$ and $g:\hom_B(b,b')$, we have the morphism
  \begin{equation*}
     G(a')_{b,b'}(g)\circ G_{a,a'}(f)_b = G_{a,a'}(f)_{b'} \circ  G(a)_{b,b'}(g)
  \end{equation*}
  in $\hom_C(G(a)(b), G(a')(b'))$.
  Functoriality is again easy to check, so we have a functor $\check{F}:A\times B \to C$.

  Finally, it is also clear that these operations are inverses.
\end{proof}

Now for any locally small precategory $A$, we have a hom-functor
\[\hom_A : A\op \times A \to \uset.\]
It takes a pair $(a,b): (A\op)_0 \times A_0 \jdeq A_0 \times A_0$ to the set $\hom_A(a,b)$.
For a morphism $(f,f') : \hom_{A\op\times A}((a,b),(a',b'))$, by definition we have $f:\hom_A(a',a)$ and $f':\hom_A(b,b')$, so we can define
\begin{align*}
  (\hom_A)_{(a,b),(a',b')}(f,f')
  &\defeq (g \mapsto (f'gf))\\
  &: \hom_A(a,b) \to \hom_A(a',b').
\end{align*}
Functoriality is easy to check.

By \autoref{ct:functorexpadj}, therefore, we have an induced functor $\y:A\to \uset^{A\op}$, which we call the \textbf{Yoneda embedding}.
As usual, of course, $\uset^{A\op}$ may not be locally small unless $A$ is small (i.e.\ unless $A_0$ lies in our fixed universe \type).

\begin{thm}[The Yoneda lemma]\label{ct:yoneda}
  For any locally small precategory $A$, any $a:A$, and any functor $F:\uset^{A\op}$, we have an isomorphism
  \begin{equation}\label{eq:yoneda}
    \hom_{\uset^{A\op}}(\y a, F) \cong Fa.
  \end{equation}
  Moreover, this is natural in both $a$ and $F$.
\end{thm}
\begin{proof}
  Given a natural transformation $\alpha:\y a \to F$, we can consider the component $\alpha_a : \y a(a) \to F a$.
  Since $\y a(a)\jdeq \hom_A(a,a)$, we have $1_a : \y a(a)$, so that $\alpha_a(1_a) : F a$.
  This gives a function $(\alpha \mapsto \alpha_a(1_a))$ from left to right in~\eqref{eq:yoneda}.

  In the other direction, given $x:F a$, we define $\alpha:\y a \to F$ by
  \[\alpha_{a'}(f) \defeq F_{a',a}(f)(x). \]
  Naturality is easy to check, so this gives a function from right to left in~\eqref{eq:yoneda}.

  To show that these are inverses, first suppose given $x:F a$.
  Then with $\alpha$ defined as above, we have $\alpha_a(1_a) = F_{a,a}(1_a)(x) = 1_{F a}(x) = x$.
  On the other hand, if we suppose given $\alpha:\y a \to F$ and define $x$ as above, then for any $f:\hom_A(a',a)$ we have
  \begin{align*}
    \alpha_{a'}(f)
    &= \alpha_{a'} (\y a_{a',a}(f))\\
    &= (\alpha_{a'}\circ \y a_{a',a}(f))(1_a)\\
    &= (F_{a',a}(f)\circ \alpha_a)(1_a)\\
    &= F_{a',a}(f)(\alpha_a(1_a))\\
    &= F_{a',a}(f)(x).
  \end{align*}
  Thus, both composites are equal to identities.
  We leave the proof of naturality to the reader.
\end{proof}

\begin{cor}\label{ct:yoneda-embedding}
  The Yoneda embedding $\y :A\to \uset^{A\op}$ is fully faithful.
\end{cor}
\begin{proof}
  By \autoref{ct:yoneda}, we have
  \[ \hom_{\uset^{A\op}}(\y a, \y b) \cong \y b(a) \jdeq \hom_A(a,b). \]
  It is easy to check that this isomorphism is in fact the action of \y on hom-sets.
\end{proof}

\begin{cor}\label{ct:yoneda-mono}
  If $A$ is a category, then $\y_0 : A_0 \to (\uset^{A\op})_0$ is a monomorphism.
  In particular, if $\y a = \y b$, then $a=b$.
\end{cor}
\begin{proof}
  By \autoref{ct:yoneda-embedding}, \y induces an isomorphism on sets of isomorphisms.
  But as $A$ and $\uset^{A\op}$ are categories and \y is a functor, this is equivalently an isomorphism on identity types, which is the definition of being mono.
\end{proof}

\begin{defn}\label{ct:representable}
  A functor $F:\uset^{A\op}$ is said to be \textbf{representable} if there exists $a:A$ and an isomorphism $\y a \cong F$.
\end{defn}

\begin{thm}\label{ct:representable-prop}
  If $A$ is a category, then the type ``$F$ is representable'' is a mere proposition.
\end{thm}
\begin{proof}
  By definition ``$F$ is representable'' is just the fiber of $\y_0$ over $F$.
  Since $\y_0$ is mono by \autoref{ct:yoneda-mono}, this fiber is a mere proposition.
\end{proof}

In particular, in a category, any two representations of the same functor are equal.
We could use this to give a different proof of \autoref{ct:adjprop} by characterizing adjunctions in terms of representability.

\section{The Rezk completion}
\label{sec:rezk}

In this section we will give a universal way to replace a precategory by a category.
It relies on the fact that ``categories see weak equivalences as equivalences''.

To prove this latter fact, we begin with a couple of lemmas which are completely standard category theory, phrased carefully so as to make sure we are using the eliminator for the propositional truncation correctly.
One would have to be similarly careful in classical category theory if one wanted to avoid the axiom of choice: any time we want to define a function, we need to characterize its values uniquely somehow.

\begin{lem}\label{lem:precomp-faithful}
  If $A,B,C$ are precategories and $H:A\to B$ is an essentially surjective functor, then $(-\circ H):C^B \to C^A$ is faithful.
\end{lem}
\begin{proof}
  Let $F,G:B\to C$, and $\gamma,\delta:F\to G$ be such that $\gamma H = \delta H$; we must show $\gamma=\delta$.
  Thus let $b:B$; we want to show $\gamma_b=\delta_b$.
  This is a mere proposition, so since $H$ is essentially surjective, we may assume given an $a:A$ and an isomorphism $f:Ha\cong b$.
  But now we have
  \[ \gamma_b = G(f) \circ \gamma_{Ha} \circ F(\inv{f})
  = G(f) \circ \delta_{Ha} \circ F(\inv{f})
  = \delta_b.\qedhere
  \]
\end{proof}

\begin{lem}\label{lem:precomp-full}
  If $A,B,C$ are precategories and $H:A\to B$ is essentially surjective and full, then $(-\circ H):C^B \to C^A$ is fully faithful.
\end{lem}
\begin{proof}
  It remains to show fullness.
  Thus, let $F,G:B\to C$ and $\gamma:FH \to GH$.
  We claim that for any $b:B$, the type
  \begin{equation}\label{eq:fullprop}
    \sm{g:\hom_C(Fb,Gb)} \prd{a:A}{f:Ha\cong b} (\gamma_a =  \inv{Gf}\circ g\circ Ff)
  \end{equation}
  is contractible.
  Since contractibility is a mere property, and $H$ is essentially surjective, we may assume given $a_0:A$ and $h:Ha_0\cong b$.

  Now take $g\defeq Gh \circ \gamma_{a_0} \circ \inv{Fh}$.
  Then given any other $a:A$ and $f:Ha\cong b$, we must show $\gamma_a =  \inv{Gf}\circ g\circ Ff$.
  Since $H$ is full, there merely exists a morphism $k:\hom_A(a,a_0)$ such that $Hk = \inv{h}\circ f$.
  And since our goal is a mere proposition, we may assume given some such $k$.
  Then we have
  \begin{align*}
    \gamma_a &= \inv{GHk}\circ \gamma_{a_0} \circ FHk\\
    &= \inv{Gf} \circ Gh \circ \gamma_{a_0} \circ \inv{Fh} \circ Ff\\
    &= \inv{Gf}\circ g\circ Ff.
  \end{align*}
  Thus,~\eqref{eq:fullprop} is inhabited.
  It remains to show it is a mere proposition.
  Let $g,g':\hom_C(Fb, Gb)$ be such that for all $a:A$ and $f:Ha\cong b$, we have both $(\gamma_a =  \inv{Gf}\circ g\circ Ff)$ and $(\gamma_a =  \inv{Gf}\circ g'\circ Ff)$.
  The dependent product types are mere propositions, so all we have to prove is $g=g'$.
  But this is a mere proposition and $H$ is essentially surjective, so we may assume $a_0:A$ and $h:Ha_0\cong b$, in which case we have
  \[ g = Gh \circ \gamma_{a_0} \circ \inv{Fh} = g'.\]

  This proves that~\eqref{eq:fullprop} is contractible for all $b:B$.
  Now we define $\delta:F\to G$ by taking $\delta_b$ to be the unique $g$ in~\eqref{eq:fullprop} for that $b$.
  To see that this is natural, suppose given $f:\hom_B(b,b')$; we must show $Gf \circ \delta_b = \delta_{b'}\circ Ff$.
  As before, we may assume $a:A$ and $h:Ha\cong b$, and likewise $a':A$ and $h':Ha'\cong b'$.
  Since $H$ is full as well as essentially surjective, we may also assume $k:\hom_A(a,a')$ with $Hk = \inv{h'}\circ f\circ h$.

  Since $\gamma$ is natural, $GHk\circ \gamma_a = \gamma_{a'} \circ FHk$.
  Using the definition of $\delta$, we have
  \begin{align*}
    Gf \circ \delta_b
    &= Gf \circ Gh \circ \gamma_a \circ \inv{Fh}\\
    &= Gh' \circ GHk\circ \gamma_a \circ \inv{Fh}\\
    &= Gh' \circ \gamma_{a'} \circ FHk \circ \inv{Fh}\\
    &= Gh' \circ \gamma_{a'} \circ \inv{Fh'} \circ Ff\\
    &= \delta_{b'} \circ Ff.
  \end{align*}
  Thus, $\delta$ is natural.
  Finally, for any $a:A$, applying the definition of $\delta_{Ha}$ to $a$ and $1_a$, we obtain $\gamma_a = \delta_{Ha}$.
  Hence, $\delta \circ H = \gamma$.
\end{proof}

The proof of the theorem itself follows almost exactly the same lines, with the saturation of $C$ inserted in one crucial step, which we have bolded below for emphasis.
This is the point at which we are trying to define a function into \emph{objects} without using choice, and so we must be careful about what it means for an object to be ``uniquely specified''.
In classical category theory, all one can say is that this object is specified up to unique isomorphism, but in set-theoretic foundations this is not a sufficient amount of uniqueness to give us a function without invoking AC.
In Univalent Foundations, however, if $C$ is a category, then isomorphism is equality, and we have the appropriate sort of uniqueness (namely, living in a contractible space).

\begin{thm}\label{ct:cat-weq-eq}
  If $A,B$ are precategories, $C$ is a category, and $H:A\to B$ is a weak equivalence, then $(-\circ H):C^B \to C^A$ is an isomorphism.
\end{thm}
\begin{proof}
  By \autoref{ct:functor-cat}, $C^B$ and $C^A$ are categories.
  Thus, by \autoref{ct:eqv-levelwise} it will suffice to show that $(-\circ H)$ is an equivalence.
  But since we know from the preceeding two lemmas that it is fully faithful, by \autoref{ct:catweq} it will suffice to show that it is essentially surjective.
  Thus, suppose $F:A\to C$; we want there to merely exist a $G:B\to C$ such that $GH\cong F$.

  For each $b:B$, let $X_b$ be the type whose elements consist of:
  \begin{enumerate}
  \item An element $c:C$; and
  \item For each $a:A$ and $h:Ha\cong b$, an isomorphism $k_{a,h}:Fa\cong c$; such that\label{item:eqvprop2}
  \item For each $(a,h)$ and $(a',h')$ as in~\ref{item:eqvprop2} and each $f:\hom_A(a,a')$ such that $h'\circ Hf = h$, we have $k_{a',h'}\circ Ff = k_{a,h}$.\label{item:eqvprop3}
  \end{enumerate}
  We claim that for any $b:B$, the type $X_b$ is contractible.
  As this is a mere proposition and $H$ is essentially surjective, we may assume given $a_0:A$ and $h_0:Ha_0 \cong b$.
  Let $c^0\defeq Fa_0$.
  Next, given $a:A$ and $h:Ha\cong b$, since $H$ is fully faithful there is a unique isomorphism $g_{a,h}:a\to a_0$ with $Hg_{a,h} = \inv{h_0}\circ h$; define $k^0_{a,h} \defeq Fg_{a,h}$.
  Finally, if $h'\circ Hf = h$, then $\inv{h_0}\circ h'\circ Hf = \inv{h_0}\circ h$, hence $g_{a',h'} \circ f = g_{a,h}$ and thus $k^0_{a',h'}\circ Ff = k^0_{a,h}$.
  Therefore, $X_b$ is inhabited.

  Now suppose given another $(c^1,k^1): X_b$.
  Then $k^1_{a_0,h_0}:c^0 \jdeq Fa_0 \cong c^1$.
  \textbf{Since $C$ is a category, we have $p:c^0=c^1$ with $\idtoiso(p) = k^1_{a_0,h_0}$.}
  And for any $a:A$ and $h:Ha\cong b$, by~\ref{item:eqvprop3} for $(c^1,k^1)$ with $f\defeq g_{a,h}$, we have
  \[k^1_{a,h} = k^1_{a_0,h_0} \circ k^0_{a,h} = \trans{p}{k^0_{a,h}}\]
  This gives the requisite data for an equality $(c^0,k^0)=(c^1,k^1)$, completing the proof that $X_b$ is contractible.

  Now since $X_b$ is contractible for each $b$, the type $\prd{b:B} X_b$ is also contractible.
  In particular, it is inhabited, so we have a function assigning to each $b:B$ a $c$ and a $k$.
  Define $G_0(b)$ to be this $c$; this gives a function $G_0 :B_0 \to C_0$.

  Next we need to define the action of $G$ on morphisms.
  For each $b,b':B$ and $f:\hom_B(b,b')$, let $Y_f$ be the type whose elements consist of:
  \begin{enumerate}[resume]
  \item A morphism $g:\hom_C(Gb,Gb')$, such that
  \item For each $a:A$ and $h:Ha\cong b$, and each $a':A$ and $h':Ha'\cong b'$, and any $\ell:\hom_A(a,a')$, we have\label{item:eqvprop5}
    \[ (h' \circ H\ell = f \circ h)
    \to
    (k_{a',h'} \circ F\ell = g\circ k_{a,h}). \]
  \end{enumerate}
  We claim that for any $b,b'$ and $f$, the type $Y_f$ is contractible.
  As this is a mere proposition, we may assume given $a_0:A$ and $h_0:Ha_0\cong b$, and each $a'_0:A$ and $h'_0:Ha'_0\cong b'$.
  Then since $H$ is fully faithful, there is a unique $\ell_0:\hom_A(a_0,a_0')$ such that $h'_0 \circ H\ell_0 = f \circ h_0$.
  Define $g_0 \defeq k_{a_0',h_0'} \circ F \ell_0 \circ \inv{(k_{a_0,h_0})}$.

  Now for any $a,h,a',h'$, and $\ell$ such that $(h' \circ H\ell = f \circ h)$, we have $\inv{h}\circ h_0:Ha_0\cong Ha$, hence there is a unique $m:a_0\cong a$ with $Hm = \inv{h}\circ h_0$ and hence $h\circ Hm = h_0$.
  Similarly, we have a unique $m':a_0'\cong a'$ with $h'\circ Hm' = h_0'$.
  Now by~\ref{item:eqvprop3}, we have $k_{a,h}\circ Fm = k_{a_0,h_0}$ and $k_{a',h'}\circ Fm' = k_{a_0',h_0'}$.
  We also have
  \begin{align*}
    Hm' \circ H\ell_0
    &= \inv{(h')} \circ h_0' \circ H\ell_0\\
    &= \inv{(h')} \circ f \circ h_0\\
    &= \inv{(h')} \circ f \circ h \circ \inv{h} \circ h_0\\
    &= H\ell \circ Hm
  \end{align*}
  and hence $m'\circ \ell_0 = \ell\circ m$ since $H$ is fully faithful.
  Finally, we can compute
  \begin{align*}
    g_0 \circ k_{a,h}
    &= k_{a_0',h_0'} \circ F \ell_0 \circ \inv{(k_{a_0,h_0})} \circ k_{a,h}\\
    &= k_{a_0',h_0'} \circ F \ell_0 \circ \inv{Fm}\\
    &= k_{a_0',h_0'} \circ \inv{(Fm')} \circ F\ell\\
    &= k_{a',h'}\circ F\ell.
  \end{align*}
  This completes the proof that $Y_f$ is inhabited.
  To show it is contractible, since hom-sets are sets, it thankfully suffices to take another $g_1:\hom_C(Gb,Gb')$ satisfying~\ref{item:eqvprop5} and show $g_0=g_1$.
  However, we still have our specified $a_0,h_0,a_0',h_0',\ell_0$ around, and~\ref{item:eqvprop5} implies both $g_0$ and $g_1$ must be equal to $k_{a_0',h_0'} \circ F \ell_0 \circ \inv{(k_{a_0,h_0})}$.

  This completes the proof that $Y_f$ is contractible for each $b,b':B$ and $f:\hom_B(b,b')$.
  Therefore, there is a function assigning to each such $f$ its unique inhabitant; denote this function $G_{b,b'}:\hom_B(b,b') \to \hom_C(Gb,Gb')$.
  The proof that $G$ is a functor is straightforward.

  Finally, for any $a_0:A$, defining $c\defeq Fa_0$ and $k_{a,h}\defeq F g$, where $g:\hom_A(a,a_0)$ is the unique isomorphism with $Hg = h$, gives an element of $X_{Ha_0}$.
  Thus, it is equal to the specified one; hence $GHa=Fa$.
  Similarly, for $f:\hom_A(a_0,a_0')$ we can define an element of $Y_{Hf}$ by transporting along these equalities, which must therefore be equal to the specified one.
  Hence, we have $GH=F$, and thus $GH\cong F$ as desired.
\end{proof}

Therefore, if a precategory $A$ admits a weak equivalence functor $A\to \widehat{A}$ where $\widehat{A}$ is a category, then that is its ``reflection'' into categories: any functor from $A$ into a category will factor essentially uniquely through $\widehat{A}$.
We now construct such a weak equivalence.

\begin{thm}\label{thm:rezk-completion}
  For any precategory $A$, there is a category $\widehat A$ and a weak equivalence $A\to\widehat{A}$.
\end{thm}
\begin{proof}
  The hom-sets of $A$ must lie in some universe \type, so that $A$ is locally small with respect to that universe.
  Write \uset for the category of sets in \type, and let $\widehat{A}_0 \defeq \setof{ F:\uset^{A\op} | \bbrck{\sm{a:A} (\y a \cong F)}}$, with hom-sets inherited from $\uset^{A\op}$.
  In other words, $\widehat{A}$ is the full subcategory of $\uset^{A\op}$ determined by the functors that are \emph{merely representable}.
  Then the inclusion $\widehat{A} \to \uset^{A\op}$ is fully faithful and a monomorphism on objects.
  Since $\uset^{A\op}$ is a category (by \autoref{ct:functor-cat}, since \uset is a category by univalence), $\widehat A$ is also a category.

  Let $A\to\widehat A$ be the Yoneda embedding.
  This is fully faithful by \autoref{ct:yoneda-embedding}, and essentially surjective by definition of $\widehat{A}_0$.
  Thus it is a weak equivalence.
\end{proof}

\begin{rmk}
  Note, however, that even if $A$ itself is a ``small category'' with respect to some universe \type (that is, both $A_0$ and all its hom-sets lie in \type), then $\widehat A$ as we have constructed it will lie in the next higher universe.
  One could imagine a ``resizing axiom'' that could deal with this.
  It is also possible to give a direct construction of $\widehat A$ using higher inductive types~\cite{ls:hits}, which leaves its universe level unchanged; see~\cite[Chapter 9]{HoTTbook}.
\end{rmk}

We call the construction $A\mapsto \widehat A$ the \textbf{Rezk completion}, although as mentioned in the introduction, there is also an argument for calling it the \textbf{stack completion}.

We have seen that most precategories arising in practice are categories, since they are constructed from \uset, which is a category by the univalence axiom.
However, there are a few cases in which the Rezk completion is necessary to obtain a category.

\begin{eg}
  Recall from \autoref{ct:fundgpd} that for any type $X$ there is a pregroupoid with $X$ as its type of objects and $\hom(x,y) \defeq \pizero{x=y}$.
  Its Rezk completion is the \emph{fundamental groupoid} of $X$.
  Under the equivalence between groupoids and 1-types, we can identify this groupoid with the 1-truncation $\trunc1X$.
\end{eg}

\begin{eg}\label{ct:hocat}
  Recall from \autoref{ct:hoprecat} that there is a precategory whose type of objects is \type and with $\hom(X,Y) \defeq \pizero{X\to Y}$.
  Its Rezk completion may be called the \emph{homotopy category of types}.
  Its type of objects can be identified with the 1-truncation of the universe, $\trunc1\type$.
\end{eg}

Finally, the Rezk completion allows us to show that the notion of ``category'' is determined by the notion of ``weak equivalence of precategories''.
Thus, insofar as the latter is inevitable, so is the former.

\begin{thm}
  A precategory $C$ is a category if and only if for every weak equivalence of precategories $H:A\to B$, the induced functor $(-\circ H):C^B \to C^A$ is an isomorphism of precategories.
\end{thm}
\begin{proof}
  ``Only if'' is \autoref{ct:cat-weq-eq}.
  In the other direction, let $H$ be $I:A\to\widehat A$.
  Then since $(-\circ I)_0$ is an equivalence, there exists $R:\widehat A\to A$ such that $RI=1_A$.
  Hence $IRI=I$, but again since $(-\circ I)_0$ is an equivalence, this implies $IR =1_{\widehat A}$.
  By \autoref{ct:isoprecat}\ref{item:ct:ipc3}, $I$ is an isomorphism of precategories.
  But then since $\widehat A$ is a category, so is $A$.
\end{proof}

\input{formalization}

\section{Conclusions and further work}
\label{sec:conclusion}

We have presented a new foundation for category theory, based on the general system of Univalent Foundations, with the following advantages:
\begin{itemize}
\item All category-theoretic constructions and proofs are automatically invariant under isomorphism of objects and under equivalence of categories (when performed with saturated categories).
\item In the rare case when we want to treat categories less invariantly, there is a separate notion available to use (strict categories).
  This allows both approaches to category theory to coexist simultaneously, with a type distinction making clear which one we are using at any given time.
\item There is a universal way to make a strict category (or, more generally, a precategory) into a saturated category, thereby passing to the invariant world in a very precise way.
  In higher-topos-theoretic semantics, this operation corresponds to the natural and well-known notion of stack completion.
\item The basic theory has all been formalized in a computer proof assistant.
\end{itemize}
One obvious direction for future work is to push forward the development of basic category theory in this system.
Another is to move on to \emph{higher} category theory: a theory of pre-2-categories and saturated 2-categories, at least, should be within reach.
Ideally, we would like a full theory of $(\infty,1)$-categories, but it has proven difficult to formalize such infinite structures in currently available type theories.

\bibliographystyle{plain}
\bibliography{hottcats}

\end{document}

%% file: formalization.tex
\section{The Formalization}
\label{sec:formalization}

Large chunks of the material presented above have been formalized in the proof assistant \textsf{Coq}.
The version of \textsf{Coq} used is \textsf{Coq} 8.3pl5, patched according to the instructions given by
Voevodsky \cite{vv_foundations}.
Our formalization is based on Voevodsky's \emph{Foundations} library~\cite{vv_foundations},
and is available online \cite{rezk_coq}. 
It is also available as an addendum to this arXiv submission.

\subsection*{Design principles}
Our general design principles largely follow the conventions established by Voevodsky \cite{vv_foundations} 
with a few departures. 
Both use only three type constructors, namely $\Pi$, $\Sigma$, $\textsf{Id}$, and avoid most of the syntactic sugar of \textsf{Coq} (such as record types).
Both do use implicit arguments and, quite extensively, coercions.

We restrict ourselves to these basic type constructors since they have a well-understood
semantics in various homotopy-theoretic models. Implicit arguments and coercions are 
crucial to manage structures of high complexity. Furthermore, they reflect
 familiar mathematical practice.

As for the differences, the use of notations, especially with infix symbols (for example, \lstinline!f ;; g! for 
the composition of morphisms of a precategory) plays an important role in our formalization. 
We also use the section mechanism of \textsf{Coq} 
when several hypotheses are common to a series of constructions and lemmas, e.g., 
when constructing particular examples of complex structures.

\subsection*{Reading the code}

Since informal type theory, used in the previous sections, is supposed to match its formal equivalent quite closely, 
the statements of the formalization are very similar to the corresponding statements of the informal type theory. 
For example, our formal statement correponding to \autoref{def:whisker} looks as follows:
\begin{lstlisting}
Lemma is_nat_trans_pre_whisker (A B C : precategory) (F : functor A B)
   (G H : functor B C) (gamma : nat_trans G H) :
  is_nat_trans (G o F) (H o F) (fun a : A => gamma (F a)).
\end{lstlisting}

The major differences occur when we split a large definition in parts as, for example, for the definition of a precategory. 
We first define:

\begin{lstlisting}
Definition precategory_ob_mor := total2 (
  fun ob : UU => ob -> ob -> hSet).
\end{lstlisting}
Given an element \lstinline!C! of the above type, we write \lstinline!a : C!
for an inhabitant \lstinline!a! of its first component (using the \emph{coercion} mechanism of \textsf{Coq}) and \lstinline!a --> b! for the value of the second 
component on \lstinline!a b : C!.

We complete the data of a precategory by:

\begin{lstlisting}
Definition precategory_data := total2 (
   fun C : precategory_ob_mor =>
     dirprod (forall c : C, c --> c)
             (forall a b c : C, a --> b -> b --> c -> a --> c)).
\end{lstlisting}
In the following we write \lstinline!identity c! for the identity morphism
on an object \lstinline!c! and \lstinline!f ;; g! for the composite of 
morphisms \lstinline!f : a --> b! and \lstinline!g : b --> c!.

We define a predicate expressing that this data constitutes a precategory:
\begin{lstlisting}
Definition is_precategory (C : precategory_data) :=
   dirprod (dirprod (forall (a b : C) (f : a --> b),
                         identity a ;; f == f)
                     (forall (a b : C) (f : a --> b),
                         f ;; identity b == f))
            (forall (a b c d : C)
                    (f : a --> b)(g : b --> c) (h : c --> d),
                     f ;; (g ;; h) == (f ;; g) ;; h).
\end{lstlisting}
As the last step, we say that a precategory is given by the data of a precategory
satisfying the necessary axioms:
\begin{lstlisting}
Definition precategory := total2 is_precategory.
\end{lstlisting}

\subsection*{Contents of the formalization}

In this part of the project we aimed on formalizing the Rezk completion together with its universal property. The formalization consists of 10 files:
\begin{itemize}
 \item \texttt{precategories.v} which roughly covers \autoref{sec:cats}.
 \item \texttt{functors\textunderscore transformations.v} which roughly covers \autoref{sec:transfors}.
 \item \texttt{sub\textunderscore precategories.v} where we define sub-precategories and
                 the image factorization of a functor. 
          This is not a separate part of the paper, but it is used (in less generality) in \autoref{thm:rezk-completion}.
 \item \texttt{equivalences.v} where we cover parts of \autoref{sec:equivalences} needed for \autoref{ct:cat-weq-eq}.
 \item \texttt{category\textunderscore hset.v} where we define the precategory of sets and show that it is a category.
 \item \texttt{yoneda.v} where we cover the main parts of \autoref{ct:yoneda}.
 \item \texttt{whiskering.v} where we define the whiskering, see \autoref{def:whisker}.
 \item \texttt{precomp\textunderscore fully\textunderscore faithful.v} that covers \autoref{lem:precomp-faithful} and \ref{lem:precomp-full}.
 \item \texttt{precomp\textunderscore ess\textunderscore surj.v} that covers \autoref{ct:cat-weq-eq}.
 \item \texttt{rezk\textunderscore completion.v} that puts the previous files together exhibiting \autoref{thm:rezk-completion}.
\end{itemize}

\subsection*{Formalization vs informal definitions}

The formalization deviates very little from the informal definitions given in the previous sections.
We shall mention here the only example of such a deviation, resulting in a slicker definition. 
In \autoref{def:adjoint} the natural transformations $(\epsilon F)$ and $(F\eta)$ (similarly, 
$(G\epsilon)$ and $(\eta G)$) are actually not
composable! We have $\epsilon F : (FG)F \to 1_{B}F$ and $F\eta : F1_A \to F(GF)$. 
However, $(FG)F$ and $F(GF)$ are not convertible, i.e.\ not \emph{definitionally} equal, 
which would be necessary for the composition to typecheck. So in order to state the equality in question
we would have to insert a transport along propositional equality---see \autoref{ct:functor-assoc} and the subsequent discussion.

We overcome this issue by rephrasing the axiom: instead of requiring an equality of natural 
transformations, we require it to hold pointwise. 
These statements are logically and type-theoretically equivalent, but for the latter we have the
desired convertibility: for any $a : A$, the term $\big(F(GF)\big)(a)$ is convertible to $\big((FG)F\big)(a)$.

\subsection*{Statistics}

Our library comprises ten files with ca.\ 180 definitions and 170 lemmas altogether.
The \texttt{coqwc} tool counts 1200 lines of specification---definitions and statements of lemmas and theorems---and 
2700 lines of proof script overall.
